\documentclass[a4paper]{amsart}

\pdfoutput=1

\usepackage[T1]{fontenc}
\usepackage[utf8]{inputenc}
\usepackage{lmodern}
\usepackage{microtype}

\usepackage[inner=3cm,outer=3cm,top=3.3cm,bottom=3cm]{geometry}

\usepackage{amssymb, amsfonts}
\usepackage{mathtools}

\usepackage{paralist}

\usepackage{xcolor}

\usepackage[english]{babel} 
\usepackage[autostyle]{csquotes} 
\usepackage{tikz, tikz-cd}

\usepackage[colorlinks=true, citecolor=blue, urlcolor=magenta, breaklinks=true]{hyperref}
\usepackage{cleveref}

\theoremstyle{plain}
\newtheorem{thm}{Theorem}[section]
\newtheorem{prop}[thm]{Proposition}
\newtheorem{cor}[thm]{Corollary}
\newtheorem{lem}[thm]{Lemma}

\theoremstyle{definition}
\newtheorem{defn}[thm]{Definition}
\newtheorem{ex}[thm]{Example}
\newtheorem{rem}[thm]{Remark}

\newcommand{\rleft}{\mathopen{}\mathclose\bgroup\left}
\newcommand{\rright}{\aftergroup\egroup\right}

\newcommand{\set}[1]{\rleft\{ {#1} \rright\}}
\newcommand{\with}{\,\colon\,}

\DeclareMathOperator{\Ann}{Ann}
\DeclareMathOperator{\cone}{cone}
\DeclareMathOperator{\conv}{conv}

\DeclareMathOperator{\Grass}{Grass}
\DeclareMathOperator{\Proj}{Proj}
\DeclareMathOperator{\Quot}{Quot}

\DeclareMathOperator{\Spec}{Spec}
\DeclareMathOperator{\Supp}{Supp}

\newcommand{\Om}{{\mathcal{O}}}

\newcommand{\kk}{\Bbbk} 
\newcommand{\CC}{\mathbb{C}}
\newcommand{\NN}{\mathbb{N}}
\newcommand{\PP}{\mathbb{P}}
\newcommand{\RR}{\mathbb{R}}
\newcommand{\ZZ}{\mathbb{Z}}

\newcommand{\PPwt}{\PP_{\mathrm{wt}}}
\newcommand{\PPwts}{\PPwt^*}
\newcommand{\PPs}{\PP^*}

\newcommand{\pp}{\mathfrak{p}}
\newcommand{\qq}{\mathfrak{q}}

\newcommand{\Rone}{R^{(1)}} 

\newcommand{\<}{\langle}
\renewcommand{\>}{\rangle}
\newcommand{\too}{\longrightarrow}

\newcommand{\hs}{h^*}

\newcommand{\Defn}[1]{\textbf{#1}}

\title{Spanning Lattice Polytopes and the Uniform Position Principle}

\author{Johannes Hofscheier}
\address{Department of Mathematics and Statistics, McMaster University, 1280 Main Street West, Hamilton, Ontario L8S4K1, Canada}
\email{johannes.hofscheier@math.mcmaster.ca}
\thanks{}

\author{Lukas Katth\"an}
\address{Institute of Mathematics, Goethe-Universit\"at Frankfurt,
	Robert-Mayer-Str. 10, 60325 Frankfurt am Main}
\curraddr{}
\email{katthaen@math.uni-frankfurt.de}
\thanks{}

\author{Benjamin Nill}
\address{Faculty of Mathematics, Otto-von-Guericke Universit\"at Magdeburg,
  Postschlie{\ss}fach 4120, 39106 Magdeburg, Germany.}
\curraddr{}
\email{benjamin.nill@ovgu.de}
\thanks{}

\subjclass[2010]{Primary: 52B20; Secondary: 13A02, 05E40, 14H45}
\keywords{Lattice polytope, Ehrhart polynomial, $h^\ast$-vector, Uniform
  Position Principle, Bertini Theorem, weighted projective space}

\begin{document}
\selectlanguage{english}

\begin{abstract}
 A lattice polytope $P$ is called {\em IDP} if any lattice point in its $k$th dilate is a sum of $k$ lattice points in $P$. In 1991 Stanley proved a strong inequality in Ehrhart theory for IDP lattice polytopes. We show that his conclusion holds under much milder assumptions, namely if the lattice polytope $P$ is {\em spanning}, i.e., any lattice point of the ambient lattice is an integer affine combination of lattice points in $P$. As an application, we get a generalization of Hibi's Lower Bound Theorem. Our proof relies on generalizing Bertini's theorem to the semistandard situation and Harris' Uniform Position Principle to certain curves in weighted projective space.
\end{abstract}

\maketitle{}

\section{Introduction}
\label{sec:intro}

\subsection{Ehrhart theory and main result}

The \Defn{Ehrhart polynomial} of a lattice polytope $P \subseteq \RR^{d}$ counts lattice points contained in dilations of the polytope, i.e., $\mathrm{ehr}_{P}(k) = \#\rleft( k P \cap \ZZ^{d} \rright)$ (see \cite{Ehr62}). Stanley showed (see \cite{Sta80}) that its generating function
\begin{equation}
  \label{hs-main}
  \sum_{k\ge0} \mathrm{ehr}_{P}(k) t^{k} = \frac{h^{*}_{P}(t)}{(1-t)^{{d+1}}}
\end{equation}
is a rational function with numerator polynomial, the so-called \Defn{$\hs$-polynomial}, whose coefficients are all nonnegative integers. Its coefficient vector is called the \Defn{$\hs$-vector} of $P$, and we denote it by $h^{*}(P)$. The maximal integer $i$ with $\hs_i(P) \not=0$ is called the \Defn{degree} of $P$. One has $\deg(P) \le \dim(P)$, with equality if and only if $P$ has a lattice point in its relative interior. The study of $\hs$-vectors of lattice polytopes is a central part of Ehrhart theory (we refer to the book \cite{Beckbook}, as well as the survey articles \cite{Beckpaper} and \cite{Braunpaper}).

In 1991 Stanley proved the following set of inequalities (see \cite[Proposition 4.1]{Stanley:CMDom}), nowadays often referred to as {\em Stanley's inequalities}.

\begin{thm}[Stanley's inequalities]
  \label{thm:known}
  Let $P$ be a lattice polytope of dimension $d$ and degree $s$. Then for $0 \leq i \leq \lfloor s/2 \rfloor$,
    \[
     \hs_0 + \hs_1 + \dotsb + \hs_i \leq \hs_{s-i} + \hs_{s - (i-1)} + \dotsb + \hs_s.
    \]
\end{thm}

Recall that a lattice polytope $P \subseteq \RR^{d}$ has the \Defn{integer decomposition property} (IDP for short) if for every lattice point $x$ in $k P$ there are lattice points $x_{1}, \ldots, x_{k}$ in $P$ such that $x = x_{1} + \ldots + x_{k}$. These polytopes are sometimes also called \Defn{integrally closed} and form a natural and much studied class of lattice polytopes (we refer to \cite{Braunpaper,CHHH}). In the same paper, Stanley gave the following stronger version of the previous theorem for IDP polytopes (see \cite[Proposition 3.4]{Stanley:CMDom}).

\begin{thm}[Stanley's stronger inequalities for IDP]
  \label{thm:idp}
  Let $P$ be an IDP lattice polytope of degree $s$. Then, for every $i,j \in \NN$ with $i + j < s$, 
  \begin{equation*}
      \hs_1 + \hs_2 + \dotsb + \hs_i \leq \hs_{j+1} + \hs_{j+2} + \dotsb + \hs_{j+i}.
  \end{equation*}
\end{thm}

The main goal of this paper is to show that this result holds under much weaker assumptions. A $d$-dimensional lattice polytope $P \subseteq \RR^d$ is called \Defn{$\ZZ^d$-spanning} (or just \Defn{spanning} if the ambient lattice is clear from the context) if every lattice point in $\ZZ^d$ is an affine integer combination of the lattice points in $P$. Clearly, IDP implies spanning.

\begin{ex}
  \label{Reeveex}
  Every $2$-dimensional lattice polytope is spanning, whereas the Reeve-simplex $\conv \rleft( 0, e_1, e_2, e_1 + e_2 + r e_3 \rright) \subseteq \RR^3$ is not spanning for $r \ge 2$. However, the Reeve-bipyramid \linebreak $\conv \rleft(-e_3, 0, e_1, e_2, e_1 + e_2 + r e_3 \rright)$ is spanning for $r \ge 2$. 
\end{ex}

The previous example also illustrates the huge qualitative difference between spanning and IDP. The Reeve-bipyramid has only $5$ lattice points but arbitrarily large volume. On the other hand, if an IDP lattice polytope has a fixed number of lattice points, then its volume must be bounded. 

Our main result is the following theorem.

\begin{thm}
  \label{thm:main}
  Let $P$ be a spanning lattice polytope of degree $s$. Then, for every $i,j \in \NN$ with $i + j < s$, 
  \begin{equation}
    \label{eq:main}
      \hs_1 + \hs_2 + \dotsb + \hs_i \leq \hs_{j+1} + \hs_{j+2} + \dotsb + \hs_{j+i}.
  \end{equation}
\end{thm}

We note that \Cref{thm:main} is false if $P$ is not spanning:
\begin{ex}
  The join of the interval $[0,3] \subseteq \RR$ and the Reeve-simplex with $r=2$ is not spanning and has $\hs$-vector $(1+2t) (1+t^2) = 1+2t + t^2 + 2t^3$ (see \cite[Lemma 1.3]{HT:LowerBounds}). Thus the special case for $i=1=j$ of \Cref{thm:main} does not hold.
\end{ex}

Yanagawa \cite{yanagawaUPP} found a further extension of \cite[Proposition 3.4]{Stanley:CMDom} for IDP polytopes. It turns out that his result does not hold for spanning lattice polytopes, see for example \cite[Example 1.5]{NOG}.

\subsection{Applications}
\label{tilde-subsec}

Let us give some applications and relate Theorem \ref{thm:main} to previously known results. For this recall that one can associate a spanning lattice polytope $\tilde{P}$ to {\em every} lattice polytope $P$: $\tilde{P}$ is given by the same vertices as $P$, however, considered as a lattice polytope with respect to the lattice spanned by the lattice points in $P$. In particular, $P$ is spanning if and only if $P = \tilde{P}$. For instance, the spanning lattice polytope associated to a Reeve-simplex (see \Cref{Reeveex}) is the unimodular simplex (a lattice simplex whose vertex set consists of the origin and a lattice basis). 

\begin{cor}
  \label{cor:spc-i1}
  For any lattice polytope $P$, we have $ \hs_1(P) \leq \hs_j(P)$ for $1 \leq j < \deg\tilde{P}$.
\end{cor}
\begin{proof}
  Since $\tilde{P}$ is obtained by coarsening the ambient lattice, we have $\hs_j(P) \geq \hs_j(\tilde{P})$ (see Section 3.2 of \cite{NOG}). As $P$ and $\tilde{P}$ contain the same number of lattice points (and thus have the same linear $\hs$-coefficient), it follows from case $i = 1$ of Theorem \ref{thm:main} that
  \[
    \hs_1(P) = \hs_1(\tilde{P}) \le \hs_j(\tilde{P}) \le \hs_j(P) \text{.} \qedhere
  \]
\end{proof}

We remark that Corollary \ref{cor:spc-i1} should be viewed as an extension to arbitrary lattice polytopes of the following celebrated result due to Hibi.

\begin{cor}[Hibi's Lower Bound Theorem \cite{Hibi94}]
  \label{cor:hibi-lbt}
  For a $d$-dimensional lattice polytope $P$ with an interior lattice point, we have $\hs_1(P) \le \hs_j(P)$ for $1 \le j < d$.
\end{cor}
\begin{proof}
  If $P$ has an interior lattice point, then so has $\tilde{P}$, and thus $\deg\tilde{P} = d$.
\end{proof}

Furthermore, Corollary \ref{cor:spc-i1} implies a new proof of the following main result of the authors' previous joint paper (we refer to \cite{NOG} for its motivation by the Eisenbud-Goto conjecture).

\begin{cor}[{\cite[Theorem 1.3]{NOG}}]
  For a spanning lattice polytope $P$, we have $1 \le \hs_j(P)$ for all $j = 0, \ldots, \deg{P}$.
\end{cor}

\begin{proof}
  If $h^*_1(P) \ge 1$, then the statement follows from Corollary~\ref{cor:spc-i1} (note that $P = \tilde{P}$ and $h^*_{\deg P} \not= 0$). If $h^*_1(P)=0$, then as $h^*_1 = |P \cap \ZZ^d| - d - 1$, the lattice polytope has precisely $d+1$ many lattice points. Hence, the spanning property implies that its vertex set forms an affine lattice basis. In particular, $h^*_P(t)=1$, so $\deg P = 0$.
\end{proof}

\subsection{Outline of the proof of Theorem~\ref{thm:main}}
\label{sec:proof}

In this section, we will give the proof of Theorem~\ref{thm:main}. We will assume knowledge of some basic facts, details can be found in the subsequent sections. 

Stanley's proof of Theorem~\ref{thm:idp} was elegant and short. It used arguments from commutative algebra and algebraic geometry, and relied on the fact that all above invariants have an interpretation in terms of algebraic invariants of the corresponding Ehrhart ring. Here, we will follow his approach, however, we will have to adapt his two main algebraic ingredients to our more general situation. 

Let $\kk$ be an algebraically closed field of characteristic $0$. The \Defn{Ehrhart ring} of a lattice polytope $P \subseteq \RR^d$ is the semigroup algebra of the lattice points in $C$, i.e., $\kk[P] \coloneqq \kk[ C \cap \ZZ^{d+1}]$, where $C = \cone \rleft( \{ 1 \} \times P \rright) \subseteq \RR^{d+1}$ is the cone over $P$. The additional coordinate in the construction of $C$ yields a natural grading on $\kk[P]$ such that the Hilbert series is exactly the generating series of the Ehrhart polynomial. By Hochster's Theorem \cite[Theorem 1]{hochster}), the Ehrhart ring is a Cohen-Macaulay domain.

In the IDP case, the Ehrhart ring is standard (generated in degree $1$), hence, the natural grading on the Ehrhart ring yields a natural embedding of $\Proj(\kk[P])$ into a projective space. By using the classical Bertini's theorem, Stanley reduces the proof to a generic hyperplane section of $\Proj(\kk[P])$. Successively, it gets reduced to the case of a curve in projective space.

For a general lattice polytope, its Ehrhart ring is only semistandard (see Definition~\ref{def:semistandard}). We need an algebraic counterpart of the spanning property of lattice polytopes. For this, let us recall that if a positively graded finitely generated $\kk$-algebra $R$ with $R_0 = \kk$ is semistandard, then there is a finite morphism $\Proj(R) \to \Proj(\Rone)$. If this morphism is even birational, then we say that $R$ is {\em virtually standard}. We refer to Definition~\ref{def:spanning} for the precise algebraic definition and to \Cref{prop:spanning} for its geometric meaning. If $P$ is spanning, then $\kk[P]$ is virtually standard (cf. \Cref{prop:ehrhartspanning}). 

Here is our first main algebraic result, its proof is the content of Section~\ref{sec:spanning} (see \Cref{thm:Bertini} and \Cref{thm:Bertini-spanning}).

\begin{thm}[Bertini for semistandard / virtually standard CM-domains] 
  \label{main-1}
  Let $R$ be a semistandard (respectively, virtually standard) Cohen-Macaulay domain of dimension $d \ge 3$ over an algebraically closed field $\kk$ of characteristic $0$. Then for a generic linear form $\theta$, $R/\langle \theta \rangle$ is again a semistandard (respectively, virtually standard) Cohen-Macaulay domain.
\end{thm}

Let $d := \dim \kk[P]$, so $\dim P = d-1$. By iterated application of \Cref{thm:Bertini}, we obtain a regular sequence $\theta_1, \dotsc, \theta_{d-2} \in \kk[P]$ of linear forms such that 
\[
  R \coloneqq \kk[P] / \<\theta_1, \dotsc, \theta_{d-2}\>
\]
is again a virtually standard domain of dimension $2$. If $R$ would be standard, it would define a curve in projective space. In \cite{HarGenus}, Harris shows that a generic hyperplane section of a curve in projective space lies in \emph{uniform position} (cf. Definition \ref{def:upp}), and in \cite{Har} he derives bounds for the Hilbert function of such a set of points. The latter are then used by Stanley to deduce the inequalities in Theorem \ref{thm:main}. In our semistandard situation, the projective variety associated to the Ehrhart ring naturally lives in a \emph{weighted} projective space where not everything behaves as nicely as for the unweighted projective space (see, for instance, \cite{Dolgachev}). More precisely, let us choose a presentation of $R$ as a quotient of a polynomial ring by a homogeneous ideal. This yields an embedding of $C \coloneqq \Proj R$ in a weighted projective space $\PPwt$ (note that several variables of the polynomial ring might have degree $>1$). Note that $C$ is a reduced irreducible curve with virtually standard homogeneous coordinate ring $R$. Let $H$ be a generic hyperplane in $\PPwt$ (i.e., defined by a generic linear form in variables of degree $1$) and set $\Gamma := C \cap H$. 

Luckily, the Uniform Position Principle carries over to this virtually standard situation. This is our second main algebraic result, its proof is the content of Section~\ref{sec:UPP} (see  \Cref{prop:upp} and \Cref{cor35}). Let us note that the result is wrong without the assumption of being virtually standard, see \Cref{ex:nonUPP}.

\begin{thm}[Uniform Position Principle for virtually standard curves]
  \label{main-2}
  For a curve $C \subseteq \PPwt$ with virtually standard homogeneous coordinate ring, a generic hyperplane section $C \cap H$ lies {\em in uniform position}. This implies that for any finite subset $\Gamma'$ of closed reduced points in $\Gamma = C \cap H$, and any integers $i,j$, one has 
  \begin{equation}\label{eq:final}
    h_{\Gamma'} ( i + j ) \geq \min( \#\Gamma', h_\Gamma(i) + h_\Gamma(j) - 1 ) \text{.}
  \end{equation}
  Here, $h_\Gamma(i)$ denotes the dimension of the $i$th-graded component of the homogeneous coordinate ring of $\Gamma$.
\end{thm}

The finish of the proof follows now exactly along Stanley's lines.  We observe from \eqref{hs-main} that 
\[
  (1-t)^{d-1} H_{\kk[P]}(t) = \frac{\hs_0 + \hs_1 t + \dotsb + \hs_s t^s}{1-t} = (h^*_0 + \cdots + h^*_s t^s) (\sum_{k=0}^\infty t^k) = \sum_{k=0}^\infty (h^*_0 + \cdots + h^*_k) t^k  \text{.}
\]
Hence, the $k$-th coefficient of the Hilbert series of $\Gamma$ (as it is defined by a regular sequence of length $d-1$) is given by $h^*_0 + \cdots + h^*_k$. As $\#\Gamma = \deg(\Gamma) = h^*_0 + \cdots + h^*_s$, a short computation shows that for $i + j \leq s$, inequality \eqref{eq:final} is equivalent to
\[
  \hs_0 + \hs_1 + \dotsb + \hs_{i+j} \geq \min(\hs_0 + \dotsb + \hs_{s}, (\hs_0 + \dotsb + \hs_{i}) + (\hs_0 + \dotsb + \hs_{j}) - 1)
\]
For $i + j < s$, the left-hand side is strictly less than the first sum of the right-hand side, and thus it has to be at least as large as the second sum. Simplification then yields 
\[
  \hs_{j+1} + \dotsb + \hs_{j+i} \geq \hs_1 + \dotsb + \hs_{i},
\]
and \Cref{thm:main} is proven. $\hfill \qed$

\subsection{What's next?}

Our main result gives further evidence that the quite large class of spanning lattice polytopes is worthwhile to study. On the one hand, it surprisingly shares properties with the much smaller class of IDP polytopes, on the other hand it avoids some of the extreme behavior of non-spanning lattice polytopes such as internal zeros in the $h^*$-vector. There is also a fascinating analogy to another general class of lattice polytopes, namely those that have interior lattice points. Building upon Hibi's Lower Bound Theorem, Stapledon studied in \cite{ALANPAPER} balanced inequalities for $h^*$-vectors of lattice polytopes with interior lattice points. Here, {\em balanced} means that the left side and the right side of an inequality contains the same number of summands. Inequality \eqref{eq:main} is such an example. Using intricate number-theoretic arguments, Stapledon was able to generate an impressive number of balanced inequalities, and even achieved a complete classification up to dimension $6$. As  \Cref{cor:spc-i1} generalizes Hibi's lower bound theorem to spanning polytopes, the big open question is now whether it is possible to prove results similar to those by Stapledon for spanning lattice polytopes.

\subsection{Organization of the paper}
In \Cref{sec:red2curves}, we describe basic properties of semistandard rings and introduce the key notion of virtually standard rings. 
In \Cref{sec:spanning}, we prove \Cref{main-1}. In Section \ref{sec:UPP}, we prove \Cref{main-2}.

\section*{Acknowledgments}

We thank Gabriele Balletti, David Eisenbud, Christian Haase and Mateusz Micha\l ek for several inspiring discussions. We are grateful to Raman Sanyal for several suggestions which resulted in improvements of the presentation of the material. Katth\"an is supported by the Deutsche Forschungsgemeinschaft (DFG), grant KA 4128/2-1. Nill is an affiliated researcher of Stockholm University and partly supported by the Vetenskapsr{\aa}det grant NT:2014-3991. Nill is a PI in the Research Training Group Mathematical Complexity Reduction funded by the German Research Foundation (DFG-GRK 2297). The work on this paper was completed while Katth\"an was in residence at the Mathematical Sciences Research Institute in Berkeley, California, during the Fall 2017 semester.

\section{Semistandard and virtually standard rings}
\label{sec:red2curves}

Throughout, let $\kk$ be a field and $R = \bigoplus_{i\ge0} R_i$ be a positively graded finitely generated $\kk$-algebra with $R_0 = \kk$. Further, denote by $\Rone = \kk[ R_1] \subseteq R$ the subalgebra generated by the elements of degree at most $1$, which is also a positively graded finitely generated $\kk$-algebra.

\subsection{Semistandard rings}

Here, we collect some basic results on semistandard rings and their associated projective varieties. 
Recall that $R$ is called \Defn{standard} if it is generated in degree $1$, i.e., if $R = \Rone$. A weaker condition due to Stanley is the following:

\begin{defn}[\cite{Stanley:CMDom}]
  \label{def:semistandard}
  A graded ring $R$ is called \Defn{semistandard} if it is integral over $\Rone$ or equivalently if $R$ is finitely generated as an $\Rone$-module (see, for instance, \cite[Corollary 4.5]{Eis}).
\end{defn}
	
\begin{rem} A lattice polytope $P$ is IDP if and only if its Ehrhart ring is standard (so, $\kk[P] = \kk[P]^{(1)}$). For non-IDP lattice polytopes, the Ehrhart ring is only semistandard, while the Ehrhart ring of a \emph{rational} polytope need not even to be semistandard.
\end{rem}

We denote by $R_+ \coloneqq \bigoplus_{i>0} R_i \subseteq R$ the \Defn{irrelevant ideal} of R. We denote by $\Proj(R)$ its associated separated scheme of relevant prime ideals. For a  homogeneous element $f \in R$ and a prime ideal $\qq$, we define as usual $D_+(f)$ as the set of relevant prime ideals not containing $f$, and $R_{(f)}$ and $R_{(\qq)}$ as the degree $0$ part of the respective localizations. We refer to \cite[Section 13.2]{GW} and \cite{hartshorne} for more details. We write $\# \mathfrak{M}$ for the cardinality of a set $\mathfrak{M}$.
\begin{lem}\label{lem:gen-lin}
  Let $R$ be semistandard.
  \begin{enumerate}
    \item Let $I \coloneqq \<R_1\>$ be the ideal generated by all elements of degree $1$. Then $\sqrt{I} = R_+$.
    \item Assume that $\#\kk = \infty$. Let $\pp_1, \dotsc, \pp_r \subseteq R$ be homogeneous prime ideals, and assume that $\pp_i \neq R_+$ for all $i$. Then a generic linear form is not contained in any of the $\pp_i$.
    \item The inclusion map $\Rone \hookrightarrow R$ induces a finite morphism $\Proj(R) \to \Proj(\Rone)$.
  \end{enumerate}
\end{lem}

\begin{proof}
  Part (1) follows from a version of Noether normalization given in \cite[Lemma I.5.2]{StanleyGreen}, but since no proof is given in \emph{loc.~cit.}, we provide one here for completeness. Note that $I \cap \Rone = R_+ \cap \Rone$. Since every other homogeneous prime ideal of $R$ is contained in $R_+$ and $R$ is an integral extension of $\Rone$, the Incomparability Theorem \cite[Corollary 4.18]{Eis} implies that $R_+$ is the only prime ideal containing $I$ and thus $\sqrt{I} = R_+$.

  Part (2) follows from part (1) using Prime Avoidance, see Lemma 3.3 and Exercise 3.19 (a) of \cite{Eis}. Part (3) follows also from part (1).
\end{proof}

All three parts of the preceding lemma easily fail if $R$ is not semistandard.

Recall the \Defn{saturation} of an ideal $I$ with respect to another ideal $J$: 
\[
  I : J^{\infty} = \set{f \in R \with \text{for all} \; g \in J \text{, there is} \; N \ge 0 \; \text{such that} \; fg^{N} \in I} \text{.} 
\]
The saturation of homogeneous ideals is again homogeneous. We remark that $( 0 : R_+^\infty )$ is also called the \emph{torsion submodule} or the \emph{$0$-th local cohomology module} $H^0_{R_+}(R)$. Note that $( 0 : R_+^\infty )$ has finite length because it is finitely generated and annihilated by a power of $R_+$. The following geometric characterization of reducedness will be used later.

\begin{lem}
  \label{lem:reduced}
  Suppose that $R$ is semistandard and let $I \subseteq R$ be a homogeneous ideal. Then $\Proj(R/I)$ is reduced if and only if $R/(I:R_+^\infty)$ is reduced.
\end{lem}
\begin{proof}
  Set $A \coloneqq R/I$ and observe that under the quotient morphism $R \to R/I$, the ideal $I : R_+^\infty$ corresponds to $0 : A_+^\infty$. Hence it suffices to show that $\Proj(A)$ is not reduced if and only if $A / ( 0 : A_+^\infty)$ is not reduced. The ``only if'' direction is straightforward to verify. So, suppose that $A' \coloneqq A / ( 0 : A_+^\infty)$ is not reduced. Let $ f \in A$ such that its residue class $\overline{f} \in A'$ does not vanish and $\overline{f}^n = 0$ for some $n > 1$. We may assume that $f$ is homogeneous (otherwise replace $f$ by its homogeneous component of highest degree). Then $\Ann(\overline{f})$, the annihilator of $\overline{f}$, is a proper homogeneous ideal in $A'$, and thus contained in a homogeneous prime ideal $\pp \subseteq A'$. Observe that $A'_+$ can't be the only prime ideal containing $\Ann(\overline{f})$ as otherwise the radical of $\Ann \rleft( \overline{f} \rright)$ would coincide with $A'_+$ implying $(A'_+)^N \overline{f} = 0$ for some $N \ge 0$, i.e., $\overline{f} = 0$ which is not possible. Hence we may take $\pp$ to be a proper subideal of $A'_+$.  Since $A'$ is semistandard, we have $A'_1 \setminus \pp \neq \emptyset$ by \Cref{lem:gen-lin}. In particular, there is $g \in A_1$ such that its residue class $\overline{g} \in A'$ is not contained in $\pp$, and thus $0 \neq f/g^{\deg(f)} \in A_{(g)}$. As $\overline{f}^n = 0 \in A'$, it follows that $f^n g^m = 0 \in A$ for some $m \ge0$, i.e., $A_{(g)}$ is not reduced. In particular $\Proj(A)$ is not reduced.
\end{proof}

The previous statement is false if $R$ is not semistandard:
\begin{ex}
  Consider $R = \kk[x,y]/\<x^2\>$ with $\deg(x) = 1$ and $\deg(y) = 2$. Let $X = \Proj(R)$. Then $R$ is not semistandard as $y$ is not integral over $R^{(1)} = \kk[x]/ \< x^2 \>$. We have 
  \[
    \Gamma( D_+(x), \Om_X) = 0 \quad \text{and} \quad \Gamma( D_+(y), \Om_X) = \kk[y]_{(y)} \text{.}
  \]
  As the open subsets $D_+(x)$ and $D_+(y)$ cover $\Proj(R)$, it follows that $\Proj(R)$ is reduced. On the other hand, one can straightforwardly verify that $\< x^2 \> : \< x, y \>^\infty = \< x^2 \>$, and thus $\kk[x,y]/(\<x^2\> : \<x, y\>^\infty) = \kk[x,y]/\<x^2\>$ is not reduced.
\end{ex}

\subsection{Multiplicity of semistandard rings}
Let us recall some facts about Hilbert functions and Hilbert series. We refer to \cite[Section 4]{BH} for more details.

The \Defn{Hilbert function} of a finitely generated graded $R$-module $M = \bigoplus_{i\ge0} M_i$ is defined as the function $i \mapsto h_M(i) := \dim_\kk M_i$, which for large $i \gg 0$ coincides with a polynomial, the \Defn{Hilbert polynomial} $P_M(t)$ (see \cite[Theorem 4.1.3]{BH}). The \Defn{Hilbert series} is the generating function
\[
  H_M(t) \coloneqq \sum_{i \geq 0} h_M(i) t^i \text{.}
\]
Recall that the (Krull) \Defn{dimension} of $M$ coincides with the longest (strictly increasing) chain of prime ideals in $\Supp(M) = \set{ \pp \in \Spec(R) \with \Ann(M) \subseteq \pp }$.

Let us now assume that $R$ is semistandard. Note that the dimension of $R$ as an $\Rone$-module coincides with its dimension as a module over itself. As $R$ is semistandard, it is a finitely generated module over the standard ring $\Rone$, and hence $M$ is also a finitely generated $\Rone$-module. Its Hilbert series can be written in the form
\[
  H_M(t) = \frac{Q(t)}{(1-t)^{\dim M}}
\]
for a polynomial $Q \in \ZZ[t]$ with $Q(1) > 0$ (see Corollary 4.1.8 and Proposition 4.1.9 of \cite{BH}). The number $e(M) := Q(1)$  is called the \Defn{multiplicity} of $M$. By \cite[Proposition 4.1.9]{BH}, if $\dim M > 0$, then the leading coefficient of $P_M$ coincides with $e(M) / (\dim(M) - 1)!$. The multiplicity of $R$ is also called the \emph{degree} of the variety $\Proj(R)$. 

\begin{ex}
  If $R = \kk[P]$ is the Ehrhart ring of a lattice polytope $P$, then $Q$ is the $\hs$-polynomial and the multiplicity is the normalized volume of $P$, i.e., $\dim(P)! \, {\rm vol}(P)$.
\end{ex}

Note that changing finitely many terms of $H_M(t)$ does not affect the multiplicity. Indeed, if $p \in \ZZ[t]$ is a polynomial, then
\[
  \frac{Q(t)}{(1-t)^{\dim M}} + p(t) = \frac{Q(t) + (1-t)^{\dim M} p(t)}{(1-t)^{\dim M}}
\]
and so the value of the numerator at $1$ does not change.

\subsection{Virtually standard rings}

Here, we introduce the algebraic counterpart of the spanning property of lattice polytopes.

\begin{defn}
  \label{def:spanning}
  We call $R$ \Defn{virtually standard} if it is semistandard and $R$ and $\Rone$ have the same multiplicity.
\end{defn}

\begin{ex}
  \label{ex:not-virt-std}
  Consider $R := \kk[x,y,z,a] / \< a^2 - x^3y\>$, with $\deg x = \deg y = \deg z = 1$ and $\deg a = 2$. It is easy to see that $\Rone = \kk[x,y,z]$ and $R \cong \kk[x,y,z] \oplus a \cdot\kk[x,y,z]$ as $\Rone$-module. Hence $R$ is semistandard and $e(R) = 2$. On the other hand, $e(\Rone) = 1$, so $R$ is not virtually standard.

  As a second example, consider $R' := R / \< xa - z^3\>$. Using \texttt{Macaulay2} \cite{M2}, one can compute that $e(R') = e((R')^{(1)}) = 6$, so this ring is virtually standard.
\end{ex}

Our definition of virtually standard rings captures the property of lattice polytopes being spanning.

\begin{prop}\label{prop:ehrhartspanning}
  The Ehrhart ring $\kk[P]$ of a lattice polytope is virtually standard if and only if $P$ is spanning.
\end{prop}
\begin{proof}
  Let $L \subseteq \ZZ^{d+1}$ be the lattice generated by $(\set{1} \times P) \cap \ZZ^{d+1}$. Set $m \coloneqq [\ZZ^{d+1} : L]$ and note that $P$ is spanning if and only if $m = 1$.  Let $R = \kk[P]$ be the Ehrhart ring associated to $P$. The subring $\Rone$ coincides with the monoid algebra corresponding to the monoid generated by the lattice points in $\set{1} \times P$. By \cite[Theorem 6.54]{BG}, it holds $e(R) = m \cdot e(\Rone)$, and thus the claim follows.
\end{proof}

\begin{rem} In the notation of Subsection~\ref{tilde-subsec}, let us note that for a lattice polytope $P$ one has $e(\kk[\tilde{P}]) = e(\kk[\tilde{P}]^{(1)})$, while $\kk[\tilde{P}] = \kk[\tilde{P}]^{(1)}$ holds if and only if $\tilde{P}$ is IDP.
\end{rem}

Recall that the \Defn{rank} of a finitely generated module $M$ over a domain $S$ is defined as the dimension $\dim_{\Quot(S)} M \otimes_{S} \Quot(S)$. In the case of a domain, we have the following  characterizations of virtually standard rings:
\begin{prop}
  \label{prop:spanning}
  For a semistandard domain $R$, the following are equivalent:
  \begin{enumerate}[(a)]
  \item $R$ is virtually standard.
  \item $R$ has rank $1$ as an $\Rone$-module.
  \item The natural inclusion of the field of fractions $\Quot(\Rone) \hookrightarrow \Quot(R)$ is an isomorphism.
  \item The natural projection $\Spec(R) \to \Spec(\Rone)$ is birational.
  \item The natural projection $\Proj(R) \to \Proj(\Rone)$ is birational.
  \end{enumerate}
\end{prop}

\begin{rem}
  In view of \Cref{prop:spanning}, a geometric interpretation for the first ring in \Cref{ex:not-virt-std} not being virtually standard is as follows: for each value of $x,y$ and $z$ with $xy \neq 0$, there are two possible values of $a$ satisfying the defining relation. In particular, $\Proj(R) \to \Proj(\Rone)$ is not birational.
  Moreover, it is clear that the rank of $R$ as an $\Rone$-module is $2$.
  
  On the other hand, the additional equation for the second ring in \Cref{ex:not-virt-std}  forces $a$ to lie in the quotient field of $(R')^{(1)}$, and thus $\Quot((R')^{(1)}) = \Quot(R')$. So this ring is virtually standard by the proposition above.
\end{rem}

We need the following lemma for the proof of \Cref{prop:spanning}.
\begin{lem}
  \label{lem:quot}
  If $R$ is a semistandard domain, then the natural map $R \otimes_{\Rone} \Quot(\Rone) \to \Quot(R)$ is an isomorphism.
\end{lem}
\begin{proof}
  It is straightforward to verify that $A \coloneqq R \otimes_{\Rone} \Quot(\Rone)$ is a domain. Moreover, since $R$ is finite over $\Rone$, it follows that $A$ is a finitely generated domain over $\Quot(\Rone)$, and thus a field.
  
  The natural map $\beta\colon A \to \Quot(R)$ is easily seen to be nonzero, and hence it is injective.  But then it is also surjective, as $\Quot(R)$ is the smallest field containing $R$.  
\end{proof}

\begin{proof}[Proof of Proposition \ref{prop:spanning}]
  The equivalence of (c) and (d) follows by the usual characterization of birational maps. The equivalence of (b) and (c) follows from \Cref{lem:quot}.
	
  Next, we show the equivalence of (a) and (c). The Hilbert series of $\Rone$ and $R$ are
  \[
    H_\Rone(t) = \frac{Q_\Rone(t)}{(1-t)^d},\quad H_R(t) = \frac{Q_R(t)}{(1-t)^d}, 
  \]
  where $d$ denotes the Krull dimension of $R$. Hilbert series are additive along short exact sequences (because vector space dimensions are), and hence the short exact sequence
  \begin{equation}\label{eq:exseq1}
    0 \too \Rone \too R \too R/\Rone \too 0
  \end{equation}
  of $\Rone$-modules implies that
  \[
    H_{R/\Rone}(t) = \frac{Q_R(t) - Q_\Rone(t)}{(1-t)^d} \text{.}
  \]	
  Note that the fraction on the right side is not necessarily reduced so that the numerator polynomial does not necessarily coincide with $Q_{R/\Rone}(t)$. The two $\Rone$-modules $R$ and $\Rone$ have the same multiplicity, i.e., $Q_R(1) - Q_\Rone(1) = 0$, if and only if $\dim (R/\Rone) < d$ (cf.~\cite[Cor. 4.1.8]{BH}) where dimensions are taken as $\Rone$-modules (see also \Cref{ex:short-ex-seq} for an illustration). Since $R$ is a domain, the latter statement is equivalent to $\Ann(R/\Rone) \neq 0$ which in turn is equivalent to $R/\Rone \otimes_{\Rone} F = 0$ (cf.~\cite[Prop 2.1]{Eis}), where $F \coloneqq \Quot(\Rone)$. Tensoring \eqref{eq:exseq1} with $F$ yields
  \begin{equation*}\label{eq:exseq2}
    0 \too F \too R \otimes_{\Rone} F \too R/\Rone \otimes_{\Rone} F \too 0.
  \end{equation*}
  We conclude that $R/\Rone \otimes_{\Rone} F = 0$ if and only if the map $F \to R \otimes_{\Rone} F$ is surjective. By Lemma \ref{lem:quot}, this is equivalent to condition (c).

  The implication (d) $\Rightarrow$ (e) is clear. We complete the proof by showing that (e) implies (c). As we are dealing with integral schemes, the morphism $\Proj(R) \to \Proj(\Rone)$ is birational if and only if it induces an isomorphism on the function fields $\kk(\Proj(\Rone)) \to \kk(\Proj(R))$ (see, for instance, \cite[Lemma 9.33]{GW}). Recall that the function field coincides with the local ring at the generic point, which in our situation in both cases is the zero ideal. In particular, the function field $\kk(\Proj(R))$ (resp.~$\kk(\Proj(\Rone))$) is the degree $0$ part $A_0$ (resp.~$A'_0$) of the ring $A$ (resp.~$A'$) obtained by inverting all homogeneous elements of $R$ (resp. $\Rone$) \cite[Proposition II.5.11]{hartshorne}. Moreover, it is well-known that $A \cong A_0[t,t^{-1}] = \kk(\Proj(R))[t,t^{-1}]$, where $t$ is an indeterminate of degree $1$ \cite[Ex. 2.17(c)]{Eis}, and similarly for $A'$.

  Hence, $A_0 \cong A'_0$ already implies that $A \cong A'$. We further conclude that $\Quot(\Rone) = \Quot(A') \cong \Quot(A) = \Quot(R)$. It is straightforward to check that this isomorphism coincides with the natural inclusion $\Quot(\Rone) \hookrightarrow \Quot(R)$.
\end{proof}

\begin{ex}
  \label{ex:short-ex-seq}
  Let $R$ be the first ring of \Cref{ex:not-virt-std}. Then $R/\Rone = a \cdot \kk[x,y,z]$ as $\Rone$-modules. In particular, $\Ann(R/\Rone) = \set{0}$ and thus $\dim(R/\Rone) = \dim(R) =\dim(\Rone) = 3$. Indeed, from the proof of \Cref{prop:spanning}, it follows that $\dim(R/\Rone) = \dim(R) = \dim(\Rone)$ if and only if $R$ is not virtually standard. We obtain $H_{R/\Rone}(t) = t^2/(1-t)^3$, so that $H_R(t) = H_\Rone(t) + H_{R/\Rone}(t)$.
\end{ex}

\section{Bertini's theorem for semistandard, respectively, virtually standard rings}
\label{sec:spanning}

In this section, we prove \Cref{main-1} (see \cite[Proposition~3.2]{Stanley:CMDom} for the standard case). It will follow from combining \Cref{thm:Bertini}(2,3) with \Cref{thm:Bertini-spanning}(1).

\begin{thm}
  \label{thm:Bertini}
  Let $R$ be a semistandard algebra over an algebraically closed field $\kk$ of characteristic $0$. For a generic linear form $\theta \in R_1$, let $S = R/ ( \<\theta\> : R_+^\infty)$. Then $S$ is semistandard and $\dim S = \max(\dim R - 1, 0)$. Moreover,
  \begin{enumerate}
  \item if $R$ is reduced, then $S$ is reduced.
  \item if $R$ is a domain and $\dim R \geq 3$, then $S$ is a domain.
  \item if $R$ is Cohen-Macaulay, then $S$ is also Cohen-Macaulay and, if in addition $\dim R \geq 2$, then $S = R / \<\theta \>$.
  \end{enumerate}
\end{thm}
We obtain Theorem \ref{thm:Bertini} from the following two classical Bertini-type theorems (see also \cite[Th\'eor\`eme I.6.3]{bertinj}).
\begin{thm}[{\cite[Theorem 1.1]{FLConn}}]
  \label{thm:irred-Bertini}
  Let $X$ be an irreducible algebraic variety over an arbitrary algebraically closed field, and $f \colon X \to \PP^m$ a morphism. Fix an integer $d < \dim \rleft( \overline{f(X)} \rright)$. Then there is a non-empty Zariski-open set $U \subseteq \Grass_{m-d} \rleft( \PP^m \rright)$ such that for all $(m-d)$-planes $H$ in $U$, $f^{-1}(H)$ is irreducible.
\end{thm}

\begin{thm}[{\cite[Theorem 1]{CGM:Bertini}}]
  \label{thm:normal-Bertini-hyperplane}
  Let $X$ be an algebraic variety over an algebraically closed field of characteristic $0$, and let $f \colon X \to \PP^m$ be a morphism. If $X$ is reduced, then there is a non-empty Zariski-open set $U \subseteq \Grass_{m-1} \rleft( \PP^m \rright)$ such that for all hyperplanes $H \in U$, $f^{-1}(H)$ is reduced.
\end{thm}

\begin{proof}[Proof of Theorem \ref{thm:Bertini}]
  The claim about the dimension follows from \Cref{lem:gen-lin}(2) and the fact that saturating by the irrelevant ideal does not change the dimension. The definition implies directly that $S$ is semistandard. Let $\dim_\kk \rleft( \Rone_1 \rright) = n$. Choose a presentation $\Rone = \kk[X_{1}, \ldots, X_{n}] / J$ of $\Rone$, where $J$ is a homogeneous prime ideal. This yields a closed embedding $\Proj(\Rone) \to \PP^{n-1}$, which we compose with the finite morphism induced by the inclusion $\Rone \to R$ (cf. \Cref{lem:gen-lin}(3)) to obtain a finite morphism $\phi \colon \Proj(R) \to \PP^{n-1}$. Let $H \subseteq \PP^{n-1}$ be a generic hyperplane defined by a linear from $\theta' \in \kk [ X_1, \ldots, X_n ]$. Then $\phi^{-1}(H)$ is defined by $\theta \coloneqq \phi^*(\theta')$ in $\Proj(R)$.
  
  \begin{asparaenum}
  \item If $R$ is reduced, then $\Proj(R)$ is reduced, and thus \Cref{thm:normal-Bertini-hyperplane} implies that $\phi^{-1}(H) = \Proj(R/\<\theta\>)$ is reduced as well. Using \Cref{lem:reduced}, we conclude that $S$ is reduced.
  \item By (1), $S$ is reduced, and thus $\< \theta \> \colon R_+^\infty$ is a radical ideal. Since $\dim R \geq 2$, we have that $\sqrt{\<\theta\>} \neq R_+$ which implies that $\<\theta\> : R_+^\infty \subseteq R_+$ (the saturation $\<\theta\> : R_+^\infty$ is homogeneous). Thus $\< \theta \> \colon R_+^\infty \subseteq \sqrt{\<\theta\>}$ and therefore it follows that $\<\theta\> \colon R_+^\infty = \sqrt{\<\theta \>}$.
    
  As $\phi$ is finite, we obtain that
  \[
    \dim \phi (\Proj(R)) = \dim \Proj(R) = \dim \Proj \rleft( \Rone \rright) = \dim \Rone -1 = \dim R - 1 \geq 2 \text{.} 
  \]
  Hence, we may apply \Cref{thm:irred-Bertini} for $d=1$ to conclude that $\phi^{-1}(H)$ is irreducible. Thus, $\sqrt{\<\theta\>} = \<\theta\>\colon R_+^\infty$ is prime.  
  
  \item If $\dim R \leq 1$, then $S$ is trivially Cohen-Macaulay. Otherwise, if $\dim R \geq 2$ and $R$ is Cohen-Macaulay, then $R_+$ is not an associated prime ideal of $R$. Hence, by \Cref{lem:gen-lin}, a generic $\theta \in R_1$ is not contained in any associated prime, and thus a nonzerodivisor. It follows that $S' := R/\<\theta\>$ is Cohen-Macaulay (see \cite[Theorem 2.1.3]{BH}). Moreover, $\dim S' \geq 1$, and hence this ring has positive depth. In particular, $(0 : (S')_+^\infty) = 0$, and thus
    \[
      R / \<\theta\> = S' = S' /(0 : (S')_+^\infty)  = R / ( \<\theta\> : R_+^\infty) = S \text{.}\qedhere
    \]
  \end{asparaenum}
\end{proof}

\begin{prop}
  \label{thm:Bertini-spanning}
  For a virtually standard ring $R$, the following rings are also virtually standard:
  \begin{enumerate}
  \item $R/ \<\theta \>$ for any nonzerodivisor $\theta \in R_1$, and
  \item $R / (0 : R_+^\infty)$. 
  \end{enumerate}
\end{prop}
\begin{proof}
  \begin{asparaenum}
  \item Since $\theta$ is a nonzerodivisor on both $R$ and $\Rone$, it holds that $H_{R / \<\theta\>}(t) = (1-t)H_R(t)$, and similarly for $\Rone/\<\theta\>$. Hence $e(\Rone/\<\theta\>) = e(\Rone) = e(R) = e(R / \<\theta\>)$.

  \item Let $T := (0 : R_+^\infty)$. Consider the map $\phi \colon \Rone \to R / T$. Its kernel is $\Rone \cap T$, its image is $(R/T)^{(1)}$, and hence it holds that $(R/T)^{(1)} \cong \Rone / (T \cap \Rone)$.
  
    Now, $T$ is a finite-dimensional $\kk$-vector space (because it is a torsion submodule of $R$), and (thus) so is $T \cap \Rone$. In particular $H_R(t)$ and $H_{R/T}(t)$ (resp.~$H_\Rone(t)$ and $H_{\Rone/(T\cap\Rone)}(t)$) differ only for a finite number of terms. It follows that
    \[
      e(R/T) = e(R) = e(\Rone) = e(\Rone / (T \cap \Rone)) = e((R/T)^{(1)}) \text{.}\qedhere
    \]
  \end{asparaenum}
\end{proof}

Finally, let us give an application of Bertini's theorem. In the standard case, it is a classical result that the multiplicity of a (reduced) projective variety equals the number of points in the intersection with a generic linear space of complementary dimension (see, for instance, \cite[p.~171]{GH}). The same holds in the semistandard setting, and since we could not locate a reference for it, we sketch a proof.

\begin{prop}
  \label{prop:deg}
  Let $R$ be a reduced semistandard algebra over an algebraically closed field $\kk$ of characteristic $0$ with $d \coloneqq \dim R$ and consider $S = R/ \rleft( \<\theta_1, \dotsc, \theta_{d-1}\> : R_+^\infty\rright)$ for generic linear forms $\theta_1, \dotsc, \theta_{d-1} \in R_1$. Then, $\Proj(S)$ is the union of exactly $e(R)$ reduced points.
\end{prop}
\begin{proof}
  We show that $e(R) = e(S)$. Set $R' \coloneqq R/ \rleft( 0 : \<\theta_1\> \rright)$ and consider the following exact sequence of graded $R$-modules where $R'[-1]$ denotes the graded module $R'$ with its grading shifted by $1$ to the left
  \[
    0 \too R'[-1] \overset{\theta_1}{\too} R \too R/ \<\theta_1\> \too 0\text{.}
  \]
  If $\theta \in R_1$ is generic, the elements annihilated by $\theta$ form a submodule of finite length (cf. \cite[Lemma 4.3.1]{HH}), and thus $H_{R'[-1]}(t) = tH_R(t) - p(t)$ for some polynomial $p(t) \in \ZZ_{\ge0}[t]$. With the exact sequence above, it follows that $H_{R/\<\theta_1\>}(t) = (1-t)H_R(t) + p(t)$. As $p(t)$ changes only finitely many terms of $H_R$, we obtain $e(R) = e(R/\<\theta_1\>)$. Similarly, saturation amounts to dividing out a torsion submodule, which has finite length, and thus it does not change the multiplicity. By induction, it follows that $e(R) = e(S)$.

  Note that $S$ is reduced by \Cref{thm:Bertini}, and thus it suffices to show that if $S$ is one-dimensional and reduced, then $\Proj(S)$ consists of $e(S)$ (reduced) points. Let $\pp_1, \dotsc, \pp_r \subset S$ denote the minimal prime ideals in $S$, so that $\Proj(S) = \set{ \pp_1, \ldots, \pp_r}$. Consider the natural homomorphism $\varphi \colon S \to S/\pp_1 \times \dotsm \times S/\pp_r \eqqcolon S'$. It is injective as $S$ is reduced. Moreover, each $S / \pp_i$ is a semistandard domain of dimension one, and thus isomorphic to a univariate polynomial ring $\kk[x_i]$ (cf. \cite[Proposition 3.1]{Stanley:CMDom}). For each degree $\ell\ge0$, $h_S(\ell) \leq h_{S'}(\ell) = r \cdot h_{\kk[x]}(\ell) = r = \#\Proj(S)$. As the Hilbert polynomial $P_S$ is a constant polynomial, it follows $P_S = e(S) \leq \#\Proj(S)$.  

  On the other hand, $(\bigcap_{i\neq j}\pp_j) \setminus \pp_i$ is not empty for $i = 1, \ldots, r$.  As we are dealing with homogeneous prime ideals, there is a homogeneous element in $\bigcap_{i\neq j} \pp_j$ which does not lie in $\pp_i$. Such elements give rise to functions vanishing on all $\pp_j$ but not on $\pp_i$ (see \Cref{ex:intpn-fct}). Hence $h_S(\ell) \geq r$ for $\ell \gg 0$, and thus $P_S = e(S) = r = \#\Proj(S)$.
\end{proof}

\begin{rem}
  \label{ex:intpn-fct}
  As this situation will also be of importance in the next section, let us recall here some basic facts in the case that $R$ is reduced with $\dim R = 1$. Furthermore, assume that $\kk$ is algebraically closed. Denote by $\pp_1, \ldots, \pp_r$ the minimal prime ideals of $R$ so that $\Gamma \coloneqq \Proj(R) = \set{ \pp_1, \ldots, \pp_r }$. Note that $\Gamma$ consists only of closed points. Fix a linear form $\theta \in R_1$ not contained in any $\pp_i$ (see \Cref{lem:gen-lin}). We note that $R_{(\pp_i)} = \kk$ for $i = 1, \ldots, r$ and that the homogeneous elements of degree $\ell \in \NN$ can be embedded in the ring of global regular functions on $\Gamma$ as follows: To $f \in R_\ell$ we associate the regular function $f/\theta^\ell \colon \Proj(R) \to \bigsqcup_{i=1}^r R_{(\pp_i)}$ and observe that if $f,g \in R_\ell$ are two homogeneous elements with the same regular function, then $f-g \in \bigcap_{i=1}^r \pp_i = \sqrt{0} = 0$.
\end{rem}

\section{The Uniform Position Principle for virtually standard curves}
\label{sec:UPP}
In this section, we will prove \Cref{main-2}. It will follow from combining \Cref{prop:upp} with \Cref{cor35}.

As a short-hand notation, for $n\in \NN$ and $a_1, \dotsc, a_r > 1$, we write
\[
  \PPwt \coloneqq \PP(\underbrace{1,\dotsc,1}_{n \text{ times}}, a_1, \dotsc, a_r) = \Proj(S) 
\]
for the \Defn{weighted projective space} with its homogeneous coordinate ring given by the polynomial ring $S = \kk[X_1,\ldots, X_n,Y_1, \ldots, Y_r]$ where $\deg X_i = 1$ and $\deg Y_j=a_j$. Note that $S^{(1)} = \kk[X_1, \ldots, X_n]$, so $\Proj(S^{(1)}) = \PP^{n-1}$. We refer to \cite{Dolgachev} for more details about weighted projective spaces.  Throughout this section, we assume $\kk = \CC$.

\subsection{Definition and results}

By a {\em curve} $C \subseteq \PPwt$, we denote an irreducible reduced closed subvariety in $\PPwt$ of dimension $1$. We say that $C \subseteq \PPwt$ is \Defn{virtually standard} if its homogeneous coordinate ring is virtually standard. In particular, by \Cref{prop:spanning}, $C \subseteq \PPwt$ is virtually standard if and only if the projection $\PPwt \dashrightarrow \PP^{n-1}$ is a birational morphism on $C$.

\begin{defn}
  \label{def:upp}
  A finite set of (closed, reduced) points $\Gamma$ in a weighted projective space \Defn{lies in uniform position} if any two subsets $\Gamma_1, \Gamma_2 \subseteq \Gamma$ with the same cardinality have the same Hilbert function.
\end{defn}

\begin{rem}
  In \cite{HarGenus}, Harris defines the uniform position property by counting conditions that subsets of $\Gamma$ impose on forms of a given degree.  This is equivalent to our definition as follows:
  \[
      \text{$\Gamma$ imposes $t$ conditions in degree $\ell$} \iff \dim_\kk (S/I(\Gamma))_\ell = t \iff h_\Gamma(\ell) = t \text{.}
  \]
\end{rem}

The following result extends Harris' \emph{Uniform Position Principle} \cite{HarGenus}.

\begin{thm}
  \label{prop:upp}
  For a virtually standard curve $C \subseteq \PPwt$, a generic hyperplane section $C \cap H$ lies in uniform position.
\end{thm}

\begin{ex}\label{ex:nonUPP}
  The assumption that $C \subseteq \PPwt$ is virtually standard cannot be dropped in \Cref{prop:upp}.  Consider a (reduced irreducible) curve $C \subseteq \PP(1,1,1,2)$ of degree $2d$, where the projection to $\PP^2$ is generically two-to-one.  Let $H$ be a generic hyperplane and $\Gamma \coloneqq C \cap H$.
	
  For any two distinct points $p_1, p_2 \in \Gamma$ with $\pi(p_1) = \pi(p_2)$, any hyperplane $H'$ passing through $p_1$ also contains $p_2$. Thus $h_{\set{p_1, p_2}}(1) = 1$.  On the other hand, if $\pi(p_1) \neq \pi(p_2)$, then one can find two linear forms vanishing on exactly one of the two points, and thus $h_{\set{p_1, p_2}}(1) = 2$.  Thus, $\Gamma$ does not lie in uniform position.
	
  An explicit example with this phenomenon is given by $C = V(xy+yz-z^2,x^2yz-a^2) \subseteq \PP(1,1,1,2)$, where $x,y,z$ have degree $1$ and $a$ has degree $2$.  Note that each point $(x,y,z)$ in the projection of $C$ to $\PP^2$ with $xyz\neq 0$ has exactly two preimages, and thus the projection $C \to C'$ is generically two-to-one. Using \texttt{Macaulay2}, one can compute that $\deg C = 4$ while $\deg C' =2$.
\end{ex}

Here is the desired application of the Uniform Position Principle on the Hilbert function. This proposition is an extension of \cite[Corollary 3.5]{Har} to the weighted setting.

\begin{prop}
  \label{cor35}
  Let $\Gamma \subseteq \PPwt$ be a finite set of closed reduced points in uniform position. Then, for any integers $i,j$, one has $h_\Gamma(i+j) \geq \min(\#\Gamma, h_\Gamma(i) + h_\Gamma(j) - 1)$.
\end{prop}

\subsection{Functions on finite sets of points}

In this section, we prove \Cref{cor35}.

The following auxiliary result follows from elementary linear algebra. We omit its straightforward proof.
\begin{lem}
  \label{lem:linal}
  Let $M$ be a finite set and write $\kk^M \coloneqq \set{f \colon M \to \kk}$ for the $\kk$-vector space of maps from $M$ to $\kk$. Fix a linear subspace $V \subseteq \kk^M$ and denote the space of restricted functions by $V|_N \coloneqq \set{f|_N \with f \in V}$ where $N \subseteq M$. Then the following holds:
  \begin{enumerate}
  \item $\dim_\kk V|_N \leq \#N$ for any $N \subseteq M$.
  \item There exists $N \subseteq M$ with $\#N = \dim_\kk V$ such that the natural map $V \to V|_N$ is an isomorphism.
  \item For $N \subseteq M$ with $\dim_\kk V|_N = \#N$ and any $N' \subseteq N$, it holds that  $\dim_\kk V|_{N'} = \#N'$.
  \item For $N \subseteq M$ with $\dim_\kk V|_N = \#N$ and any $p \in N$, there exists a function $f \in V$ with $f(p) \neq 0$ which vanishes on $N \setminus\set{p}$.
  \end{enumerate}
\end{lem}

Let $\Gamma \subseteq \PPwt$ be a finite set of (closed, reduced) points and write $R$ for its homogeneous coordinate ring. We can interpret the homogeneous elements of $R$ as functions on $\Gamma$ (see  \Cref{ex:intpn-fct}). So, for the homogeneous coordinate ring $R'$ of any subset $\Gamma' \subseteq \Gamma$, we have that the projection $R \to R'$ corresponds to the restriction of regular functions on $\Gamma$ to $\Gamma'$. Using this identification, the following is immediate from \Cref{lem:linal}:
\begin{lem}
  \label{lem:points}
  Let $\Gamma \subseteq \PPwt$ be a finite set of (closed, reduced) points, and let $\ell \in \NN$. Then the following holds:
  \begin{enumerate}
  \item $h_{\Gamma'}(\ell) \leq \#\Gamma'$ for any $\Gamma' \subseteq \Gamma$.
  \item There exists $\Gamma' \subseteq \Gamma$ with $\#\Gamma' = h_{\Gamma}(\ell) = h_{\Gamma'}(\ell)$.
  \item Let $\Gamma' \subseteq \Gamma$ such that $h_{\Gamma'}(\ell) = \#\Gamma'$. Then for any $\Gamma'' \subseteq \Gamma'$, it holds that  $h_{\Gamma''}(\ell) = \#\Gamma''$.
  \end{enumerate}
\end{lem}

Next, we give a folklore characterization of sets in uniform position.

\begin{prop}
  \label{lem:charUPP}
  For a finite set $\Gamma \subseteq \PPwt$ of closed reduced points, the following are equivalent:
  \begin{enumerate}
  \item $\Gamma$  lies in uniform position.
  \item For each subset $\Gamma' \subseteq \Gamma$, it holds that $h_{\Gamma'}(\cdot) =  \min(h_{\Gamma}(\cdot), \#\Gamma')$.
  \end{enumerate}
\end{prop}
\begin{proof}
  The implication (2) $\Rightarrow$ (1) is trivial. It remains to show (1) $\Rightarrow$ (2). Let $\Gamma' \subseteq \Gamma$ be a subset and $\ell \in \NN$. By \Cref{lem:points}(2), there exists a subset $\Gamma'_0 \subseteq \Gamma$ with $h_{\Gamma'_0}(\ell) = \# \Gamma'_0 = h_\Gamma(\ell)$. We distinguish three cases:
  \begin{itemize}
  \item $\#\Gamma' = h_\Gamma(\ell)$: As $\# \Gamma' = \# \Gamma'_0$, we have $h_{\Gamma'} = h_{\Gamma_0'}$, and thus $h_{\Gamma'}(\ell) = h_{\Gamma}(\ell) = \# \Gamma'$.
  \item $\#\Gamma' > h_\Gamma(\ell)$: For any $\Gamma'' \subseteq \Gamma'$ of size $h_\Gamma(\ell)$, we have $h_{\Gamma''} = h_{\Gamma'_0}$, and thus $h_\Gamma(\ell) \geq h_{\Gamma'}(\ell) \geq h_{\Gamma''}(\ell) = h_{\Gamma'_0}(\ell) = h_\Gamma(\ell)$.
  \item $\#\Gamma' < h_\Gamma(\ell)$: In this case $\Gamma'$ is contained in a set $\Gamma'' \subseteq \Gamma$ of size $h_\Gamma(\ell)$. The claim follows from $h_{\Gamma''} = h_{\Gamma'_0}$ and part (3) of \Cref{lem:points}. \qedhere
  \end{itemize}  
\end{proof}

Now, we can prove \Cref{cor35}.

\begin{proof}[Proof of \Cref{cor35}]
  Let $R$ be the homogeneous coordinate ring of $\Gamma$. As in \Cref{ex:intpn-fct}, we consider the homogeneous elements of $R$ as functions on $\Gamma$. Choose a subset $\Gamma' \subseteq \Gamma$ of size $\min(\#\Gamma, h_\Gamma(i) + h_\Gamma(j) - 1)$. Fix a point $p$ in $\Gamma'$ and choose two further subsets $V, W \subseteq \Gamma'$ such that $V \cap W = \set{p}$, $V \cup W = \Gamma'$, $\#V \leq h_{\Gamma}(i)$ and $\#W \leq h_{\Gamma}(j)$.
	
  Since $\Gamma$ lies in uniform position, by \Cref{lem:charUPP}, it holds that $h_V(i) = \#V$ and $h_W(j) = \#W$. Hence, part (4) of \Cref{lem:linal} implies that there is a function $r \in R_i$ of degree $i$ which is nonzero on $p$ and zero on all other points of $V$. Similarly, there is a function $g \in R_j$ of degree $j$ which vanishes on all points of $W$ except $p$. The product $fg \in R_{i+j}$ provides a function which vanishes on all points of $\Gamma'$ expect $p$.
  
  Since we can do this for each point of $\Gamma'$, and the resulting functions are by construction linearly independent, it follows that 
  \[
    h_\Gamma(i+j) \geq h_{\Gamma'}(i+j) \geq \#\Gamma' = \min(\#\Gamma, h_\Gamma(i) + h_\Gamma(j) - 1). \qedhere
  \]
\end{proof}

\subsection{Hyperplane sections of virtually standard curves}
In this section, we are going to prove \Cref{prop:upp}. Recall that the homogeneous coordinate ring of $\PPwt$ is denoted by $S$. Let us write $\PPwts$ for the set of hyperplanes in 
$\PPwt$. So, $\PPwts$ can be identified with $\PP(S_1)$, i.e., the space of linear forms on $\PPwt$ up to scaling. Moreover, the isomorphism $S_1 \cong S_1^{(1)}$ yields an identification of $\PPwts$ with $\PPs$, the set of hyperplanes in $\Proj( S^{(1)} ) = \PP^{n-1}$.

We prove \Cref{prop:upp} by reducing it to the unweighted case. The crucial step is the following lemma whose statement is a direct generalization of the corresponding result for curves in ordinary projective spaces (see \cite[p. 197]{HarGenus}).

\begin{lem}
  \label{lem:inccor}
  Let $C \subseteq \PPwt$ be a virtually standard curve. For each $t \in \NN$ with $1 \leq t \leq \deg C$, the incidence correspondence
  \[
    I_t \coloneqq \set{(H, p_1, \dotsc, p_t) \with H \in \PPwts, p_1, \dotsc, p_t \in C \cap H \text{ distinct}}
  \]
  is irreducible of dimension $\dim \Proj( S^{(1)} ) = n-1$.
\end{lem}
\begin{proof}
  First of all, note that by \Cref{prop:deg}, a generic hyperplane section of $C$ contains $\deg C \ge t$ points, which implies that $I_t$ is not empty.
	
  We claim that all irreducible components of $I_t$ have dimension $n-1$. Since the projection to $\PPwts$ has finite fibers, the dimension cannot be larger. Let $X_1, \ldots, X_k$ be the irreducible components of $I_t$ and fix $i = 1, \ldots, k$. Consider a point $(H, p_1, \dotsc, p_t)$ in $X_i$ not contained in $X_j$ for $j \neq i$. For a small analytic neighborhood $U \subseteq \PPwts$ of $H$, every $\tilde{H} \in U$ intersects $C$ in (at least) $t$ distinct points, and hence the preimage of $U$ in $I_t$ is a disjoint union of several copies of $U$, one of which contains $(H, p_1, \dotsc, p_t)$. As $X_i \setminus (\bigcup_{j\neq i} X_j)$ is Zariski open in $I_t$, it follows that the dimension of $X_i$ is $n-1$.
	
  The final step of the proof is to show that $I_t$ is irreducible. Denote the homogeneous coordinate ring of $C$ by $R$ and set $C' \coloneqq \Proj(\Rone) \subseteq \PP^{n-1}$, so we have the finite morphism $C \to C'$.  Consider the incidence correspondence of $C'$:
  \[
    I'_t \coloneqq \set{(H, p'_1, \dotsc, p'_t) \with H \in \PPs, p'_1, \dotsc, p'_t \in C' \cap H \text{ distinct}}
  \]
  Note that $\deg C' = \deg C$ as $R$ is virtually standard, and hence $I'_t$ is also not empty. By \cite[p. 197]{HarGenus} (see also \cite[Corollary 6.7.8]{DS}), $I'_t$ is irreducible with $\dim I'_t = n-1$.
	
  The finite morphism $\pi \colon C \to C'$ (cf.~Lemma \ref{lem:gen-lin}) is birational as $R$ is virtually standard (see \Cref{prop:spanning}). To be more precise, there are finitely many points $q_1, \ldots, q_r$ in $C$ resp.~$q'_1,\ldots,q'_s$ in $C'$ such that the restriction of $\pi$ to the open subsets obtained by removing those points is an isomorphism. Let $V \subseteq I_t$ and $V' \subseteq I'_t$ be the open subsets obtained by considering only those hyperplanes in $\PPwt^*$ avoiding $q_1, \ldots, q_r$ and $q'_1, \ldots, q'_s$, respectively. Then $\pi$ induces an isomorphism $\tilde{\pi} \colon V \to V'$. Since $\dim (I_t \setminus V) < n-1$, the (Zariski) open subset $X_i \cap V$ is dense in $X_i$ for all $i$. Moreover, $I'_t$ is irreducible, and hence $\tilde{\pi}(X_i \cap V)$ is dense in $I'_t$. But since $\tilde{\pi}$ is an isomorphism, we conclude that the intersection of two irreducible components contains a (non-empty) open subset, and thus $I_t$ is irreducible.
\end{proof}

The following lemma is well-known for ordinary projective spaces. We provide a proof in the weighted case for completeness.
\begin{lem}
  \label{lem:minor}
  Let $Z \subseteq \PPwt$ be the subspace of points whose first $n$ coordinates are zero. Then, for $t,\ell, k \in \NN$ the set
  \[
    U(t,\ell,k) \coloneqq \set{(p_1, \dotsc, p_t) \in (\PPwt\setminus Z)^t \with h_{\set{p_1, \dotsc, p_t}}(\ell) < k}
  \]
  is closed in $(\PPwt\setminus Z)^t$.
\end{lem}
\begin{proof}
  For each tuple $(p_1, \dotsc, p_t) \in (\PPwt\setminus Z)^t$, there exists a (generic) hyperplane $H \subseteq \PPwt$ which contains none of the $p_i$ (note that $Z$ is exactly the set of all points contained in every hyperplane). Hence, we can cover $(\PPwt\setminus Z)^t$ by open sets of the form
  \[
    U_H \coloneqq \set{(p_1, \dotsc, p_t) \in (\PPwt\setminus Z)^t \with p_1, \dotsc, p_t \notin H } \text{.}
  \]
  It is sufficient to show that $U(t,\ell,k) \cap U_H$ is closed in $U_H$ for all $H$. So let us fix a hyperplane $H$.
	
  Recall that the homogeneous coordinate ring of $\PPwt$ is denoted by $S$. Let $\theta \in S_1$ be the linear form corresponding to $H$, so the functions of degree $\ell$ on $\set{p_1, \dotsc, p_t}$ arise as restrictions of rational functions $f / \theta^\ell$ for $f \in S_\ell$ to $\set{p_1, \dotsc, p_t}$.   Let $f_1, \dotsc, f_r \in S_\ell$ be a $\kk$-basis of $S_\ell$. Now,  $h_{\set{p_1, \dotsc, p_t}}(\ell)$ can be computed as the rank of a matrix $M$ where the $(i,j)$-th entry of $M$ is the evaluation of $f_i / \theta^\ell$ at $p_j$. Therefore, the locus where $h_{\set{p_1, \dotsc, p_t}}(\ell) < k$ is the locus where the rank of $M$ is less than $k$, and thus closed.
\end{proof}

Now, we are ready to prove the Uniform Position Principle.
\begin{proof}[Proof of \Cref{prop:upp}]
  Let $H \in \PPwts$ be a generic hyperplane. As $C \cap H$ is finite, its homogeneous coordinate ring is $1$-dimensional. By \cite[Theorem 4.1.3]{BH}, there exists a degree $\ell_0 \in \NN$ such that for all $\ell \geq \ell_0$, it holds that $h_{H\cap C}(\ell) = \#(C \cap H)= \deg C$. Note that $h_{C \cap H}(\ell) = h_C(\ell) - h_C(\ell - 1)$, and therefore $\ell_0$ does not depend on $H$.
  
  For any subset $\Gamma \subseteq C \cap H$ and $\ell \geq \ell_0$, it follows from part (3) of \Cref{lem:points} that $h_\Gamma(\ell) = \#\Gamma$, so in this range the Hilbert function depends only on the cardinality of $\Gamma$.

  To finish the proof, we need to show that this is also true for $\ell < \ell_0$. Fix two integers $1 \leq \ell \leq \ell_0$ and $1 \leq t \leq \deg C$ and consider the incidence correspondence $I_t$ from Lemma \ref{lem:inccor}. Let $J_{t,\ell, k}$ be the set of points $(H, p_1, \dotsc, p_t) \in I_t$ where $h_{\set{p_1, \dotsc, p_t}}(\ell) < k$. As $R$ is semistandard, we have that $C \subseteq \PPwt\setminus Z$, and thus the image of the projection $I_t \to \PPwt^t$ is contained in $(\PPwt\setminus Z)^t$ (here we used the notation of \Cref{lem:minor}, note that $Z$ is exactly the locus where the projection $\PPwt \dashrightarrow \PP^{n-1}$ is not defined). As $J_{t,\ell, k}$ coincides with the preimage of $U(t,\ell, k)$ under that map, it is closed. Note that $J_{t,\ell,t+1} = I_t$ and $J_{t,\ell,0} = \emptyset$ and let $0 \le k_0 = k(t, \ell) \le t+1$ be the maximal number such that $J_{t,\ell,k_0}$ is a proper subset of $I_t$. Since $I_t$ is irreducible by \Cref{lem:inccor}, we have $\dim J_{t,\ell,k_0} < \dim I_t = n-1 = \dim \PPwts$. Hence, the image of $J_{t,\ell,k_0}$ under the projection $I_t \to \PPwts$ misses a dense open subset $U_{t, \ell}$. Note that, by construction, for each $H \in U_{t, \ell}$ and any $\Gamma \subseteq C \cap H$ of size $t$, it holds that $h_{\Gamma}(\ell) = k_0$. The nonempty set $U := \bigcap_{t=1}^{\deg C}\bigcap_{\ell=1}^{\ell_0} U_{t, \ell}$ is open in $\PPwt^*$ such that  $C \cap H$ lies in uniform position for each $H \in U$.
\end{proof}

\providecommand{\bysame}{\leavevmode\hbox to3em{\hrulefill}\thinspace}
\providecommand{\href}[2]{#2}

\end{document}